\documentclass[12pt]{amsart}
\usepackage[margin=1in]{geometry}

\title[Macdonald symmetry at $q=1$ and inv-preserving bijections]{Macdonald symmetry at $q=1$ and a new class of inv-preserving bijections on words}

\author{Maria Gillespie, Ryan Kaliszewski, and Jennifer Morse}


\usepackage{amsmath}
\usepackage{amsfonts}
\usepackage{amssymb}
\usepackage{amsthm}
\usepackage{graphicx}
\usepackage{xypic}
\usepackage{tikz}
\usepackage{tabmac_avi}

\newcommand{\ZZ}{\mathbb{Z}}

\newcommand{\Sn}{\mathfrak S_n}

\newtheorem{theorem}{Theorem}[section]

\newtheorem{corollary}[theorem]{Corollary}

\newtheorem{proposition}[theorem]{Proposition}
\newtheorem{question}[theorem]{Question}

\theoremstyle{definition}
\newtheorem{definition}{Definition}
\newtheorem{example}{Example}
\newtheorem{remark}{Remark}

\DeclareMathOperator{\inv}{inv}
\DeclareMathOperator{\Inv}{Inv}

\DeclareMathOperator{\maj}{maj}

\newcommand{\elevator}{\psi}
\newcommand{\foata}{\varphi}
\newcommand{\tH}{\widetilde{H}}
\renewcommand{\th}{^{\text{th}}}

\begin{document}

\maketitle

\begin{abstract}We give a direct combinatorial proof of the $q,t$-symmetry relation $\tH_{\mu}(X;q,t)=\tH_{\mu'}(X;t,q)$ in the Macdonald polynomials $\tH_\mu$ at the specialization $q=1$.  The bijection demonstrates that the Macdonald $\inv$ statistic on the permutations of any given row of a Young diagram filling is Mahonian.  Moreover, our bijection gives rise a family of new bijections on words that preserves the classical Mahonian $\inv$ statistic.
\end{abstract}

\section*{Introduction}

Two of the best known statistics on words $w=w_1w_2\cdots w_n$ on an ordered alphabet are the major index ($\maj$) and inversion number ($\inv$), defined by
\[ 
\maj(w)=\sum_{w_d>w_{d+1}} d\;,
  \qquad 
\inv(w)=|\{(i,j)\mid i<j\text{ and }w_i>w_j\}|\;.
\]
The major index was introduced by MacMahon \cite{MacMahon13}, and the inversion number was well studied even by that time.  MacMahon proved that the major index and inversion number are \emph{equidistributed}, that is their generating functions over any word $w$ are equal:
\[ 
\sum_{w'\text{ a rearrangement of $w$}} q^{\maj(w')} 
=
\sum_{w'\text{ a rearrangement of $w$}} q^{\inv(w')}\;.
\]
We refer to any statistic that is equidistributed with inv (or maj) as \textit{Mahonian}.

A variant of both of these statistics arises in the combinatorial formula for the Macdonald polynomials $\tH_\mu(x_1,x_2,\ldots;q,t)$, originally defined in a slightly different form in \cite{Macdonald}.  They form an orthogonal basis of the ring of symmetric functions $\Lambda_{\mathbb{Q}(q,t)}(X)=\Lambda_{\mathbb{Q}(q,t)}(x_1,x_2,x_3,\ldots)$, and are a $q,t$-analog of the homogeneous symmetric functions.  

Macdonald polynomials appear in a wide range of mathematical areas such as the representation theory of quantum groups \cite{Fomin97} and double affine Hecke algebras \cite{Cherednik95}, the Calogero-Sutherland model in particle physics \cite{particle}, the ring of diagonal harmonics and the shuffle algebra of symmetric functions \cite{HHLRU}, the geometry of Hilbert schemes \cite{Haiman01}, Springer fibers of Hikita \cite{hikita12}, 
and the elliptic Hall algebra of Shiffmann-Vasserot \cite{ehall}.  They are known to be Schur positive \cite{Haiman01} and are therefore the Frobenius characters of certain doubly graded $\Sn$-modules $\mathcal{H}_\mu$ \cite{GarsiaHaiman}.  

Most progress in understanding the structure of Macdonald polynomials arose from generalizations of Lascoux and Sch\"utzenberger's work on the Hall-Littlewood polynomials \cite{LS}, which have been shown to be a special case of the Macdonald polynomials.  Characterizing Macdonald polynomials themselves remained non-trivial until 2004 when Haglund \cite{Haglund04} gave a simple combinatorial formula for computing the Macdonald polynomials,
\begin{equation} \tH_\mu(X;q,t)=\sum_{F} q^{\inv(F)}t^{\maj(F)}X^F \end{equation}
where the sum is over combinatorial fillings of shape $\mu$.

In the formula above, the major index is defined on the columns of the filling, mirroring MacMahon's definition on words.  However, the inversion statistic counts inversion pairs and inversion triples.  Inversion pairs are inversions (in the standard sense) appearing in the first row.  Inversion triples appearing within a fixed row are believed to be Mahonian, but a direct proof of this eluded researchers.

Using algebraic means, Haiman proved that Macdonald polynomials satisfy a conjugate $q,t$-symmetry:
\begin{equation} \tH_\mu(X;q,t)=\tH_{\mu'}(X;t,q) \end{equation}
where $\mu'$ is the conjugate partition.
A natural question arose about whether a combinatorial proof of this fact could be found using Haglund's formula.  So far no such proof has been found.  

However, great strides have been made in several important special cases.  Gillespie \cite{Gillespie16} gave a combinatorial proof for hook shapes $\mu$, for certain infinite families of shapes $\mu$ in the Hall-Littlewood specialization $t=0$, and for the coefficient of $x_1\cdots x_n$ in the $t=0$ specialization for all shapes.  Gillespie also showed that
\begin{equation} \label{Macdonaldq=1} \tH_\mu(X;1,t) = \tH_{\mu'}(X;t,1) \end{equation}
by using a generalization of the Carlitz bijection on permutations \cite{carlitz75} and an equidistribution argument.   
Independently, Kaliszewski and Morse \cite{KaMo} have proven Equation \eqref{Macdonaldq=1} by using a new Mahonian statistic and quasisymmetric functions.  
Even in this case, however, a direct, bijective proof of Equation \eqref{Macdonaldq=1} on fillings has eluded researchers.

In this paper we will show that the inversion statistic, when restricted to a single row of a combinatorial filling, is Mahonian by constructing a bijection between words called the \textit{elevator map}.  We will then use this bijection together with Foata's bijection \cite{foata68}, $\foata:W_n\to W_n$, or any other such bijection that satisfies
\[ \maj(w) = \inv(\foata(w)), \]
to bijectively prove Equation \eqref{Macdonaldq=1}.  (See Section \ref{section:Macdonald}.)  The elevator map additionally has several interesting consequences and specializations in its own right, and we explore a number of these consequences in Sections \ref{section:liftmaps} and \ref{section:elevator} below.

\section{Relative inversions and lift maps} \label{section:liftmaps}

Let $W_n$ be the set of all words of length $n$ from the alphabet $\ZZ_+$.  For $w\in W_n$ let $w_i$ denote the $i\th$ letter of $w$ and $w_{[i,j]}$ be the subword
\[ w_{[i,j]}=w_iw_{i+1}\ldots w_j. \]
Recall that for a word $w\in W_n$, an \textit{inversion} is a pair $(i,j)$ such that $i<j$ and $w_i>w_j$.  Let the set of inversions of $w$ be denoted $\Inv(w)$ and $\inv(w) = |\Inv(w)|.$  We can generalize the notion of an inversion as follows.

\begin{definition}
For $k\in\ZZ_{\geq 0}$ a \textit{relative $k$-inversion}, or simply $k$-inversion, of a word $w\in W_n$ is a pair $(i,j)$ such that $i<j$ and either $k\ge w_i> w_{j}$, $w_i> w_{j}>k$, or $w_{j}> k\ge w_i$. In other words, $w_i\neq w_j$, $w_j\neq k$, and there is a cyclic shift of the sequence $(k,w_i,w_j)$ in which the entries are nonincreasing.  

Denote the set of $k$-inversions of $w$ by $\Inv^k(w)$ and let 
\[ \inv^k(w)=|\Inv^k(w)|. \]
\end{definition}

Note that a $0$-inversion is the same as a usual inversion.

\begin{example}
  The word $25132$ has $2$-inversions $(1,2)$ (corresponding to the first and second entries), $(1,3)$, $(1,4)$, $(2,4)$, $(3,4)$.  Thus $\inv^2(25132)=5$. 
\end{example}

\begin{remark}
  The definition of a $k$-inversion can be realized as a standard inversion under the ordering $\prec_k$ defined by
\begin{equation} \label{korder} k+1\prec_k k+2\prec_k\cdots\prec_k n\prec_k 1\prec_k\cdots\prec_k k. \end{equation}
\end{remark}

\begin{definition}
  For each $i\in \mathbb{Z}_+$ we define the $i\th$ \textit{basement lift map} $B_i:W_n\to W_n$ by $(B_i(w))_j=i$ if and only if $w_{n+1-j}=i$, and the subword $u$ of $w$ consisting of all entries not equal to $i$ is equal to the subword $v$ of $B_i(w)$ consisting of all entries not equal to $i$.
\end{definition}

\begin{example}
Let $w$ be the word $22{\color{red}34}2{\color{red}3313}2$, then
\[ B_2(w) = 2{\color{red}3433}2{\color{red}13}22. \]
The entries not equal to 2 have been highlighted read so that we can easily see that the positions of the 2's have been reversed and the other letters remain in the same order.
\end{example}

\begin{proposition} \label{invprop}
 If $w\in W_n$, we have
 \begin{equation} \inv^k(w) = \inv^{k+1}(B_{k+1}(w)). \end{equation}
\end{proposition}
\begin{proof}
 Suppose that $(i,j)$ is a $k$-inversion of $w$ with $w_j\neq k+1$.  This means that $w_i\succ_k w_j$, but since $w_j\neq k+1,$ $w_i\succ_{k+1}w_j$ as well, and so $(i,j)$ is a $(k+1)$-inversion in $w$.  Since $B_{k+1}$ preserves the relative order of every letter with the exception of $k+1$'s, the corresponding pair in $B_{k+1}(w)$ will still be a $k+1$-inversion of $B_{k+1}(w)$.  This argument is reversible, so $(i,j)$ with $w_j\neq k+1$ is a $k$-inversion of $w$ if and only if the corresponding pair $(\hat\imath,\hat\jmath)$ is a $k+1$-inversion of $B_{k+1}(w)$.
 
Now suppose $w_j=k+1$.  The number of $k$-inversions $(i,j)$ in $w$ is equal to the number of entries $w_i$ to the left of $w_j$ for which $w_i\neq k+1$.  In addition, for every $i>j$, $(j,i)$ will not be a $k$-inversion.

In $B_{k+1}(w)$, the number of $(k+1)$-inversions $(\hat\imath,\hat\jmath)$ with $(B_{k+1}(w))_{\hat\imath}=k+1$ is equal to the number of entries to the right of $(B_{k+1}(w))_{\hat\imath}$ not equal to $k+1$.  In addition, for every $\hat\jmath<\hat\imath$, $(\hat\jmath,\hat\imath)$ will not be a $k+1$-inversion.  Since the action of $B_{k+1}$ is to reverse the positions of all of the $k+1$'s, it follows that we have $\inv^k(w)=\inv^{k+1}(B_{k+1}(w))$. \end{proof}

We also require certain compositions of basement lift maps.

\begin{definition}
  For $j>i$, we write $B_{i\to j}=B_{j}\circ B_{j-1}\circ \cdots \circ B_{i+2}\circ B_{i+1}$.  If $j>i$ we also write $B_{j\to i}=B_{i\to j}^{-1}=B_{i+1}^{-1}\circ B_{i+2}^{-1}\cdots \circ B_{j}^{-1}$.
\end{definition}

It is interesting to note that for $w\in W_n$ with maximal entry $m$,
\[ \inv(w)=\inv^0(w)=\inv^m(w)=\inv^k(w) \]
for any $k>m$.  So if $W_n^m$ is the set of words of length $n$ with maximal entry $m$, we have a nontrivial map $B_{0\to m}:W_n^m\to W_n^m$ that preserves $\inv$, i.e.\begin{equation}\label{eqn:basementattic} \inv(w) = \inv(B_{0\to m}(w)). \end{equation}  Furthermore, this map preserves the content of the word in question.

\subsection{Restriction to permutations}

The basement lifting maps are particularly interesting on permutations.  Throughout this section we think of the symmetric group $\Sn$ as the subset of $W_n$ consisting of words with distinct entries $1,\ldots,n$.  For a permutation $w\in \Sn$, the effect of $B_i$ is simply to reflect the position of the $i$ in w and keep the ordering of the remaining elements the same.  It is not hard to see that by Proposition \ref{invprop} and Equation \ref{eqn:basementattic}:

\begin{corollary}
 The basement lift maps $B_i$ are bijections $B_i:\Sn\to \Sn$. If $w\in \Sn$, we have $$\inv^k(w) = \inv^{k+1}(B_{k+1}(w))\;,
  \qquad  
  \inv(w) = \inv(B_{0\to n}(w))\;.$$
\end{corollary}

\begin{example}
Let $\sigma\in\mathfrak S_5$ be the permutation $5,1,3,2,4$.  The image $B_{0\to n}(\sigma)$ would be computed

\[ \begin{tikzpicture}
    \node (A) at (0,0) {51324};
    \node (B) at (2,0) {53214};
    \node (C) at (4,0) {53214};
    \node (D) at (6,0) {52134};
    \node (E) at (8,0) {45213};
    \node (F) at (10,0) {42153.};
    \draw[->] (A) -- (B) node[pos=.5,above] {$B_1$};
    \draw[->] (B) -- (C) node[pos=.5,above] {$B_2$};
    \draw[->] (C) -- (D) node[pos=.5,above] {$B_3$};
    \draw[->] (D) -- (E) node[pos=.5,above] {$B_4$};
    \draw[->] (E) -- (F) node[pos=.5,above] {$B_5$};
    
\end{tikzpicture} \]
Note that $\inv(51324)=\inv(42153)=5$.
\end{example}

There are many natural questions that arise from this new $\inv$-preserving map on permutations, in particular:

\begin{question}\label{question:orbits} 
 The map $B_{0\to n}$ is an action on $\Sn$.  What can be said about the orbits?
\end{question}

\begin{question}\label{question:foata}
 Can we conjugate by the Foata bijection in order to obtain a simple $\maj$-preserving bijection on permutations?  Similarly, is there a simple description of the map induced by $B_{0\to n}$ on Carlitz codes (see \cite{carlitz75})? 
\end{question}

\section{The elevator map} \label{section:elevator}

In this section we  generalize the basement lifting map to ``arbitrary basements'' of two-row Macdonald fillings.

\begin{definition}
  Let $m\geq n$ and $a\in W_m$.  The \textit{elevator map} $\elevator_{a}:W_{n}\to W_{n}$ is defined by the following process.  Starting with a word $w=w^{(0)}\in W_n$, let $w^{(1)}=B_{0\to a_1}(w)$ and define the words $w^{(2)},\ldots, w^{(n)}$ by setting 
\[ w^{(i)} = w^{(i-1)}_{[1,i-1]}\cdot B_{a_{i-1}\to a_i}\left(w^{(i-1)}_{[i,n]}\right). \]
That is, $w^{(i)}$ is the result of applying $B_{a_{i-1}\to a_i}$ to the last $n-i+1$ letters of $w^{(i-1)}$ (and leaving the first $i-1$ letters the same).  For all $i>n$, define $w^{(i)}=w^{(n)}$.
Let $\elevator_a(w)=w^{(n)}$.
\end{definition}

\begin{example}
Let $a=2122321$ and $w=321332$.  First we set $w^{(0)}=w$ and compute 
\[ w^{(1)}=B_{0\to 2}(w^{(0)})=B_2\circ B_1(w^{(0)})=B_2(323132)=233123. \]
To construct $w^{(2)}$ by applying $B_{2\to 1}$ on the last 5 letters of $w^{(1)}$,
\[ w^{(2)}=2\cdot B_{2\to 1}(33123)=2\cdot B_2(33123)=232313. \]
Repeating this process,
\[ w^{(3)}=23\cdot B_{1\to 2}(2313)=23\cdot B_2(2313)=233132,
\]
\[ w^{(4)}=233\cdot B_{2\to 2}(132)=233132,
\]
\[ w^{(5)}=2331\cdot B_{2\to 3}(32)=2331\cdot B_3(32)=233123,
\]
\[ w^{(6)}=23312\cdot B_{3\to 2}(3)=23312\cdot B_3(3)=233123,
\]
\[ w^{(7)}=233132\cdot B_{2\to 1}(\emptyset)=233132\cdot B_2(\emptyset)=233123.
\]
Thus $\elevator_{2122321}(321332)=233123$.
\end{example}

\subsection{Macdonald Fillings}

A \textit{filling} is a map $F$ from the cells of a Young diagram of a partition $\mu$ to the positive integers.  We express a filling as a diagram in which each cell $c$ contains its image, $F(c)$.  The shape of a filling is the shape of the underlying Young diagram.

\begin{example}
A filling of shape $(3,3,1)$ is shown below.
\[ \tableau[scY]{ 4|1,2,2|3,1,3}\;. \]
\end{example}

In \cite{Haglund04}, Haglund gave a description of the Macdonald polynomials as generating functions for fillings;
\begin{equation} \label{haglundformula} \tH_{\mu}(X;q,t)=\sum_{F \text{ filling of shape $\mu$}} q^{\inv(F)}t^{\maj(F)}X^F, \end{equation}
where
\[ X^F = \prod_{c \text{ cell in $F$}} x_{F(c)}. \]
This revolutionary formula gave researchers a computationally tractable formula for $\tH_\mu$.  The generalized statistics $\inv$ and $\maj$ on fillings are defined as follows.  Recall that 
on words, $\maj(w) = \sum_{w_d>w_{d+1}} d$.  For a column $C$ of a filling $F$, define the \textit{column word} $r(C)$ by reading the entries of $C$ from top to bottom.  Then the {\it major index} of $F$ is
\[ \maj(F) = \sum_{C \text{ column of $F$}} \maj(r(C)). \]

\begin{example}
If
\[ F = \tableau[scY]{ 3 | {\color{red} 4},1,{\color{red} 3} | 1,{\color{red} 2},2 | 3,1,3 }\;, \]
then $\maj(F)=5$.  The descents in each column are highlighted in red; the first occurring in position $2$ from the top, the second in position $2$, and the third in position $1$.
\end{example}

For a filling $F$, an {\it inversion triple} is a triple of cells
\[ \tableau[scY]{x, \cdots, y | z} \]
with $x=z\neq y$, $x>y>z$, $y>z>x$, or $z>x>y$.  In other words, $x\neq y, y\neq z$ and there is a cyclic shift of the sequence $(x,y,z)$ in which the entries are nonincreasing.  We say that such an inversion triple occurs in row $i$ where $x$ and $y$ are in the $i$th row from the bottom.

An {\it inversion pair} is a pair of cells
\[ \tableau[scY]{x,\cdots,y} \]
appearing in the first row with $x>y$.
The {\it inversions} of $F$ are the inversion triples and inversion pairs, i.e.
\[ \inv(F) = \#\text{inversion triples} + \#\text{inversion pairs}. \]

\begin{remark} If we imagine that every cell lying below the filling contains the entry 0, then an inversion pair is actually an inversion triple with $z=0$.
\end{remark}

\begin{remark} To use the language of section \ref{section:liftmaps}, a triple of cells is an inversion triple if and only if the $x$ and $y$ form a $z$-inversion.  That is, if we impose the order that
\[ z+1\prec_z z+2\prec_z \ldots \prec_z k \prec_z 1 \prec_z \ldots \prec_z z \]
for $k$ the maximal entry in $F$,
\[ \tableau[scY]{x, \cdots, y | z} \]
is an inversion triple if and only if $y\prec_z x$.
\end{remark}

\begin{example}
The previous example has one inversion pair and two inversion triples (in red):
\[ \tableau[scY]{ 3 | 4,1,3 | 1,2,2 | {\color{red} 3}, {\color{red} 1},3 }\;, \qquad 
\tableau[scY]{ 3 | {\color{red} 4},1,{\color{red} 3} | {\color{red} 1},2,2 | 3,1,3 }\;,
    \qquad
\tableau[scY]{ 3 | 4,{\color{red} 1},{\color{red} 3} | 1,{\color{red} 2},2 | 3,1,3 }\;.\]
Therefore, $\inv(F)=3$.
\end{example}

\subsection{The elevator map and fillings}

We now extend the elevator map to Macdonald fillings.

\begin{definition}\label{def:fillingmap}
Let $F$ be a filling and $R_1,R_2,\ldots,R_k$ be the rows of $F$ from bottom to top.  To extend the elevator map to fillings, define $\elevator(F)$ to be the filling with rows $R_1',\ldots,R_k'$ defined recursively by
\[ R'_i = \left\{ 
  \begin{array}{cl}
    R_1 & \text{if i=1} \\     
    \elevator_{R'_{i-1}}(R_i) & \text{otherwise} \\
  \end{array}
  \right. \;.
\]
\end{definition}

\begin{example}
If
\[ F = \tableau[scY]{ 3,3,1 | 2,2,3 | 1,2,1 } \]
to compute $\elevator(F)$ we compute the rows one at a time.  The first row is untouched,
\[ \tableau[scY]{ ,, | ,, | 1,2,1 }\;. \]
We then fill the second row with $\elevator_{R_1'}(R_2)=\elevator_{121}(223)=322$,
\[ \tableau[scY]{ ,, | 2,3,2 | 1,2,1 }\;. \]
Finally, fill the third row with $\elevator_{R_2'}(R_3)=\elevator_{232}(331)=133$,
\[ \tableau[scY]{ 1,3,3 | 2,3,2 | 1,2,1 }\;. \]
\end{example}

\begin{theorem}
If $F$ is a filling with rows $R_1,R_2,\ldots,R_k$ then
\[ \sum_{i=1}^k \inv(R_i) = \inv(\elevator(F)). \]
\end{theorem}

\begin{proof}
It is sufficient to prove that, in the notation of Definition \ref{def:fillingmap}, the number of inversion triples (or pairs) in row $R_i'$ of $\elevator(F)$ is equal to $\inv(R_i)$ for all $i$.  We proceed by induction.  For $R_1$, this is clearly true since $R_1'=R_1$ and the inversion pairs remain the same.  

Now, let $m>1$ and assume the claim is true for all $i\le m-1$.  We wish to show that the number of inversion triples of in $R_m'$ is equal to $\inv(R_m)$.

Recall that to compute $\elevator_{R_{m-1}'}(R_m)$, we construct a sequence of words $$R_m=w^{(0)}, w^{(1)},w^{(2)},\ldots,w^{(k)}=R_m'.$$  Construct a corresponding sequence $F^{(0)},\ldots,F^{(k)}$ of fillings of a two-row diagram of shape $(\mu_{m-1},\mu_m)$ where for $i\ge 1$ the first row of $F^{(i)}$ is the first $i$ letters of $R_{m-1}'$ followed by repeated copies of $R_{m-1}'$, and the second row is $w^{(i)}$, and if $i=0$, in $F^{(0)}$, the first row contains all zeroes.  (See Example \ref{example:tworow} below.) The second rows of $F^{(0)},F^{(1)},\ldots,F^{(k)}$ are $w^{(0)}, w^{(1)},w^{(2)},\ldots,w^{(k)}$ respectively.  We claim that each $F^{(i)}$ has the same number of inversion triples in the second row.

It is clear that $\inv(R_m)=\inv(F^{(0)})$.  If row $R_{m-1}'$ has entries $a_1,\dots,a_{\mu_{m-1}}$, then
\begin{align*}
\inv(F^{(1)}) 
 & = \inv^{a_1}(B_{0\to a_1}(R_{m})) \\
 & = \inv^{a_1}(B_{a_1}\circ B_{a_1-1}\circ\cdots\circ B_1(R_m)) \\
 & = \inv^{a_1-1}(B_{a_1-1}\circ\cdots\circ B_1(R_m))\\
 & = \inv^0(R_{m})=\inv(R_m) \\
 & = \inv(F^{(0)}),
\end{align*}
by repeated application of Proposition \ref{invprop}.

By induction, assume that $F^{(i)}$ has the same number of inversion triples as $F^{(j)}$ for all $j<i$.
Constructing $w^{(i+1)}$ from $w^{(i)}$ only permutes the last $n+1-i$ letters, so any letter that was part of an inversion triple with cells in position $i$ or less will still be an inversion triple.  

Since the first row of $F^{(i)}$ is completely filled with $a_i$'s after position $i$, any inversion triple in cells completely after position $i$ is equivalent to an $a_i$-inversion in the second row.  We know the number of $a_i$-inversions is equal to the number of $a_{i+1}$-inversions after applying $B_{a_i\to a_{i+1}}$ by Proposition \ref{invprop}.  Changing all of the entries in the first row after position $i$ to $a_{i+1}$ allows us to describe the inversion triples as $a_{i+1}$-inversions.  Thus $\inv(F^{(i)})=\inv(F^{(i+1)})$, completing the proof.
\end{proof}

\begin{example}\label{example:tworow}
Consider the filling
\[ F=\tableau[scY]{4,3,1,4 | 1,4,3,2}\;. \]
Each of the following fillings has three inversion triples:
\[ F^{(0)}=\tableau[scY]{4,3,1,4 | 0,0,0,0}\;, \]
\[ F^{(1)}=\tableau[scY]{4,1,3,4 | 1,1,1,1}\;, \]
\[ F^{(2)}=\tableau[scY]{4,4,1,3 | 1,4,4,4}\;, \]
\[ F^{(3)}=\tableau[scY]{4,4,1,3 | 1,4,3,3}\;, \]
\[ F^{(4)}=\tableau[scY]{4,4,1,3 | 1,4,3,2}\;, \]
\end{example}

\section{The symmetry $H_\mu(X;1,t)=H_{\mu'}(X;t,1)$}\label{section:Macdonald}

Before we can prove Equation \eqref{Macdonaldq=1} bijectively, we require one more ingredient: a bijection $f:W_n\to W_n$ for which $\inv(f(w))=\maj(w)$ for all $w\in W_n$.  The {\it Foata bijection}   $\foata:W_n\rightarrow W_n$ \cite{foata68} and the generalized {\it Carlitz bijection} on words $\theta:W_n\rightarrow W_n$ \cite{carlitz75,Gillespie16} are two of the most well-known such maps.

We will now construct a bijection from fillings of $\mu$ to fillings of $\mu'$ that sends $\maj$ to $\inv$, giving a direct bijective proof of Equation \eqref{Macdonaldq=1}.  For a filling $F$ of shape $\mu$, let $C_1,C_2,\ldots,C_k$ be the columns of $F$.  From the definition we have that
\[ \maj(F) = \sum_i \maj(C_i). \]

If we apply the Foata map $\foata$ (or the Carlitz map $\theta$, or any other such map) to each column, we obtain a collection of words $R_1,\ldots,R_k$ with $\inv(R_i) = \maj(C_i)$, so
\[ \maj(F) = \sum_i \inv(R_i).\] 
Now, construct the filling $F'$ of shape $\mu'$ whose rows are $R_1,\ldots,R_k$.  Then apply the elevator map $\elevator(F')$ to obtain a filling such that 
\[ \inv(\elevator(F')) = \sum_i \inv(R_i) = \sum_i \maj(C_i)=\maj(F). \]

Since every stage of the construction $F\mapsto \elevator(F')$ is a bijection, this gives a bijection between fillings of shape $\mu$ and fillings of shape $\mu'$ such that
\[ \inv(\elevator(F')) = \maj(F). \]

\begin{theorem}\label{thm:symmetry1} The bijection above gives a direct combinatorial proof that, for any partition $\mu$,
  \[ \tH_\mu(X;1,t) = \tH_{\mu'}(X;t,1). \]
\end{theorem}

\subsection{Future work}

The surprising results of Sections \ref{section:liftmaps} and \ref{section:elevator} motivate two important new directions of further study.  

The first is to better understand the implications of the elevator map, combined with the Carlitz or Foata bijections, on our combinatorial understanding of Macdonald polynomials.  In particular, the monomials appearing in $\tH_\mu(X;1,t)$ are in one to one correspondence with those appearing in $\tH_{\mu}(X;q,t)$, and so Theorem \ref{thm:symmetry1} may be extendable to the general $q,t$-symmetry problem.  

Additionally, the $q=0$ specialization, $\tH_{\mu}(X;0,t)$, has a well-understood Schur function expansion in terms of the \textit{cocharge} statistic of Lascoux and Sch\"utzenberger \cite{LS}, which can be proven to be essentially equivalent to the Macdonald $\maj$ statistic.  However, its symmetric counterpart, $\tH_{\mu'}(X;t,0)$, is not yet understood in terms of an analogous statistic arising from $\inv$.  We hope that our new proof of symmetry in the $q=1$ case can be specialized to the $q=0$ case - which, combinatorially, involves simply restricting to certain sets of filligns - in order to obtain a new Schur expansion for the Hall-Littlewood polynomials.

The second natural direction of future research is to better understand the composition $B_{0\to n}$ of basement lifting maps on permutations and words that results in an $\inv$-preserving bijection, as in Questions \ref{question:orbits} and \ref{question:foata}.  This nontrivial bijection is worthy of study in its own right, and yet we can study its variants in the context of Macdonald polynomials as well.  For instance, the nontrivial powers of $B_{0\to n}$ can be composed with the elevator map to provide alternative bijections that similarly prove Theorem \ref{thm:symmetry1}. 

\section*{Acknowledgements}
Thanks to Jake Levinson for his help translating the abstract, and to Anne Schilling, Jim Haglund, and John Berman for helpful conversations over the course of this work.

\bibliographystyle{alpha}
\bibliography{elevatorbib}

\end{document}